\documentclass[11pt, oneside]{amsart}   	
\usepackage{geometry}                		
\usepackage{graphicx}				
\usepackage{amssymb}

\usepackage{pst-node}
\usepackage{tikz-cd} 
\usepackage{mathtools}

\usepackage{mathrsfs}

\usepackage{amsthm}

\usepackage{amsfonts}
\usepackage{yhmath}
\usepackage{cleveref}
\usepackage{caption}

\theoremstyle{definition}
\newtheorem{definition}{Definition}[section]

\theoremstyle{remark}
\newtheorem{remark}[definition]{Remark}

\theoremstyle{theorem}
\newtheorem{theorem}[definition]{Theorem}

\theoremstyle{corollary}
\newtheorem{corollary}[definition]{Corollary}

\theoremstyle{lemma}
\newtheorem{lemma}[definition]{Lemma}

\theoremstyle{example}
\newtheorem{example}[definition]{Example}

\theoremstyle{prop}
\newtheorem{prop}[definition]{Proposition}
\author{}
\makeatletter
\@namedef{subjclassname@2020}{%
	\textup{2020} Mathematics Subject Classification}
\makeatother
\begin{document}
\title{Filtrations and torsion pairs in Abramovich-Polishchuk's heart}

\author{Yucheng Liu}

\address{College of Mathematics and Statistics, Center of Mathematics, Chongqing University, Chongqing, 401331, China}
\email{noahliu@cqu.edu.cn}

\keywords{Bridgeland stability conditions, product varieties, filtrations}

\subjclass[2020]{14F08, 14J40, 18E99}
\maketitle 
\begin{abstract}
We study abelian subcategories and torsion pairs in Abramovich-Polishchuk's heart. And we apply the construction from \cite{Stabilityconditionsonproductvarieties} on a full triangulated subcategory $\mathcal{D}_S^{\leq 1}$ in $D(X\times S)$, for an arbitrary smooth projective variety $S$. We also define a notion of $l$-th level stability, which is a generalization of the slope stability and the Gieseker stability. We show that for any object $E$ in Abramovich-Polishchuk's heart, there is a unique filtration whose factors are $l$-th level semistable, and the phase vectors are decreasing in a lexicographic order.
	
\end{abstract}
	\section{Introduction}
	The theory of Bridgeland stability conditions was introduced by Bridgeland in \cite{bridgeland2007stability}, motivated by Douglas's work on D-branes and $\Pi$-stability (see \cite{douglas2002dirichlet}). This theory was further studied by Kontsevich and Soilbelman (see \cite{kontsevich2008stability}), and it turns out to be related to many different branches of mathematics including algebraic geometry (\cite{bayer2014projectivity}, \cite{bayer2014mmp}), symplectic geometry (\cite{Conjecturesonstabilitycondition}, \cite{Stabilityconditionsonsymplectictopology}), representation theory (\cite{kontsevich2008stability}, \cite{Scatteringdiagrams}), and curve counting theories (\cite{CurvecountingtheoriesviastableojectsI}, \cite{CurvecountingtheoriesviastableobjectsII}).

	Abramovich-Polishchuk's construction of sheaves of $t$-structures (\cite{APsheaves},\cite{polishchuk2007constant}) becomes a fundamental tool in the study of Bridgeland stability conditions (see \cite{Stabilityconditionsinfamilies}) and moduli problems of complexes (\cite{bayer2014projectivity}). More explicitly, given a $t$-structure with Noetherian heart on the bounded derived category of coherent sheaves $D(X)$, with $X$ a smooth variety over a field $k$ of characteristic $0$, Abramovich and Polishchuk (\cite{APsheaves}) constructed a $t$-structure on $D(X\times S)$, where $S$ is another smooth variety over the same base field $k$. Their construction was refined by Polishchuk in his paper \cite{polishchuk2007constant}, where he removed the characteristic $0$ and smoothness assumption. 
	
	Their construction has important applications on the study of Bridgeland stability conditions. For instance, Bayer and Macr\`i used it to construct a nef divisor on the moduli space of Bridgeland semistable objects (see \cite{bayer2014mmp} and \cite{bayer2014projectivity}); the author used their construction to construct Bridgeland stability conditions on $X\times S$ (see \cite{Stabilityconditionsonproductvarieties}), where $S$ is a smooth projective curve.
	
	In this paper, we will further study the relation between Abramovich-Polishchuk's heart (AP heart for short) and stability conditions. As in \cite{Stabilityconditionsonproductvarieties}, suppose that we have a stability condition $\sigma=(\mathcal{A},Z)$ on $D(X)$ (see Definition \ref{first WSC} below for the definition of stability conditions), with $\mathcal{A}$ the Noetherian heart of a bounded $t$-structure on $D(X)$ and $Z:K_0(\mathcal{A})\rightarrow \mathbb{C}$ a central charge with discrete image. We denote the corresponding AP heart on $D(X\times S)$ by $\mathcal{A}_S$. Then for any object $E\in\mathcal{A}_S$, there is a complexified Hilbert polynomial $$L_E(n)=a_r(E)n^r+a_{r-1}(E)n^{r-1}+\cdots +a_0(E)+i(b_r(E)n^r+b_{r-1}(E)n^{r-1}\cdots+b_0(E))$$ where $r= dim(S)$ and $a_i, b_i$ are linear maps from the Grothendieck group $ K_0(D(X\times S))$ to $\mathbb{R}$ for any $0\leq i\leq r$.

	For any integer $0\leq j\leq r$, one can define the full subcategory $\mathcal{A}_S^{\leq j} $ in $\mathcal{A}_S$ consisting of objects whose complexified Hilbert polynomial is of degree $\leq j$. It is easy to show that $\mathcal{A}_S^{\leq j}$ is an abelian subcategory. And there is a unique filtration $$0=E_0\subset E_1\subset\cdots\subset E_r=E$$ for any object $E\in\mathcal{A}_S$, such that $E_j$ is the maximal subobject of $E$ in $\mathcal{A}_S^{\leq j}$. If we use $\mathcal{D}^{\leq 1}_S$ to denote the minimal triangulated subcategory generated by $\mathcal{A}_S^{\leq 1}$ in $D(X\times S)$, we can apply the construction in \cite{Stabilityconditionsonproductvarieties} on $\mathcal{D}_S^{\leq 1}$. Thus we get the following result.
	
	\begin{theorem}\label{main corollary}
		The triangulated category $\mathcal{D}_S^{\leq 1}$ admits stability conditions.
	\end{theorem}

	In the meantime, we prove a sequence of positivity results on the coefficients of $L_E(n)$ (see also \cite[Lemma 4.1]{somequadraticinequalities}). These positivity results help us to construct weak stability functions on $\mathcal{A}_S$ and its abelian subcategories. This lead us to the notion of  $l$-th level stability, which can be viewed as a generalization of both the slope stability and the Gieseker stability. However, the phase of an object with respect to such stability is no longer a single real number, but a vector of real numbers. Then we show the existence and uniqueness of the filtration of an object $E\in\mathcal{A}_S$ with respect to such stability, which we call the lexicographic order filtration of $E$. 
\begin{theorem}\label{lexi-order filtration in introduction}
		Assume $l\leq j$, given $t_1,t_2,\cdots,t_l\in\mathbb{Q}_{>0}$, for any object $E\in\mathcal{A}_S^{\leq j}$, there exists a  unique filtration $$0=E_0\subset E_1\subset E_2\cdots \subset E_n=E,$$such that each quotient factor $E_{i}/E_{i-1}$ is $l$-th level semistable of phase $\vec{\phi}_i$ and $$\vec{\phi}_1>\vec{\phi}_2>\cdots>\vec{\phi}_n.$$
\end{theorem}	

Given $\vec{t}=(t_1,\cdots,t_{r+1})\in\mathbb{Q}_{>0}^{r+1}$ and a phase vector $\vec{\phi}=(\phi_1,\cdots,\phi_{r+1})\in(0,1)^{r+1}$, using this theorem, we can construct torsion pairs on $\mathcal{A}_S$:$$\mathcal{T}_{\vec{\phi}}^{\vec{t}}\coloneqq \{E\in\mathcal{A}_S|all \ the \ factors \ in \ theorem\ \ref{lexi-order filtration in introduction}\ have \ phases \ > \vec{\phi}\},$$

$$\mathcal{F}_{\vec{\phi}}^{\vec{t}}\coloneqq\{E\in\mathcal{A}_S|all \ the \ factors \ in \ theorem\ \ref{lexi-order filtration in introduction}\ have \ phases \leq \vec{\phi}\}.$$

In the end of this paper, we state some positivity results on the tilted heart $\mathcal{A}_{S,\vec{\phi}}^{\vec{t}}\coloneqq \langle \mathcal{F}_{\vec{\phi}}^{\vec{t}}[1], \mathcal{T}_{\vec{\phi}}^{\vec{t}}\rangle$ (see Theorem \ref{positivity on tilted heart}).
	
	\subsection*{Outline of the paper} In Section \ref{BSC}, we review some basic definitions and results in the theory of stability conditions. In Section \ref{Preliminary results}, we review some necessary results from \cite{APsheaves}, \cite{polishchuk2007constant} and \cite{Stabilityconditionsonproductvarieties}. In Section \ref{Abelian subcategories}, we study some abelian subcategories on $\mathcal{A}_S$ and state some general results on weak stability conditions. In Section \ref{torsion pairs}, we apply the results from previous section on $\mathcal{A}_S$. As a result, we define the notion of $l$-th level stability and prove the existence and uniqueness of lexicographic order filtration. Moreover, we use the filtration to construct torsion pairs of $\mathcal{A}_S$ and show some positivity results on the tilted hearts.
		\subsection*{Notation and Conventions} In this paper, we work over an algebraically closed field in arbitrary characteristic. All varieties are integral separated algebraic schemes of finite type over such a field $k$. We will use $D(X)$ rather than the usual notation $D^b(cohX)$ to denote the bounded derived category of coherent sheaves on $X$. We use $\mathbb{H}$ to denote the upper half plane, i.e., the set of complex numbers with positive imaginary part. All categories are assumed to be essentially small in this paper. 
	
	\subsection*{Acknowledgement} I would like to thank Hanfei Guo, Emanuele Macr\`i, Hao Sun, Zhiyu Tian, and Qizheng Yin for helpful discussions. I am very grateful to the referee for a careful reading of the manuscript and many helpful comments. The author is financially supported by NSFC grant 12201011.
	
		\section{Bridgeland Stability conditions}\label{BSC}

		In this section, we review the definition and some basic results on (weak) stability conditions (See \cite{beilinson1982faisceaux}, \cite{polishchuk2007constant}, \cite{bridgeland2008stability}, \cite{kontsevich2008stability} and \cite{baye2011bridgeland}).
		
		\begin{definition}
			Let $\mathcal{D}$ be an triangulated category. A $t$-structure on $\mathcal{D}$ is a pair of full subcategories $(\mathcal{D}^{\leq 0},\mathcal{D}^{\geq 0})$ satisfying the conditions (i)-(iii) below. We denote $\mathcal{D}^{\leq n}=\mathcal{D}^{\leq 0}[-n]$, $\mathcal{D}^{\geq n}=\mathcal{D}^{\geq 0}[-n]$ for every $n\in\mathbb{Z}$. Then the conditions are:
			
			(i) $Hom(X,Y)=0$ for every $X\in\mathcal{D}^{\leq 0}$ and $Y\in\mathcal{D}^{\geq 1}$;
			
			(ii) $\mathcal{D}^{\leq -1}\subset \mathcal{D}^{\leq 0}$ and  $\mathcal{D}^{\geq 1}\subset \mathcal{D}^{\geq 0}$.
			
			(iii) every object $X\in\mathcal{D}$ fits into an exact triangle 
			
			$$\tau^{\leq 0}X\rightarrow X\rightarrow \tau^{\geq 1}X\rightarrow \cdots$$ with $\tau^{\leq 0}X\in\mathcal{D}^{\leq 0}$, $\tau^{\geq 1}X\in\mathcal{D}^{\geq 1}$.
			
		It is not hard to see that the objects $\tau^{\leq 0}X$ and $\tau^{\geq 0}X$ are defined functorially. The functors $\tau^{\leq 0}$ and $\tau^{\geq 1}\coloneqq \tau^{\geq 0}[-1]$ are called truncation functors (see e.g. \cite[Proposition 8.1.8]{hotta2007d} for the basic properties of truncation functors). 	The heart of a $t$-structure is $\mathcal{A}=\mathcal{D}^{\leq 0}\cap\mathcal{D}^{\geq 0}$. It is an abelian category (see e.g. \cite[Theorem 8.1.9]{hotta2007d}). We have the associated cohomology functors defined by $H^0=\tau^{\leq 0}\tau^{\geq 0}:\mathcal{D}\rightarrow\mathcal{A}$, $H^i(X)\coloneqq H^0(X[i])$. 
		\end{definition}
		Motivated by the work from String Theory (see e.g. \cite{douglas2002dirichlet}), the notion of t-structures was refined by Bridgeland to the following one in \cite{bridgeland2007stability}.
		
		\begin{definition}\label{slicing}
			A slicing on a triangulated category $\mathcal{D}$ consists of full subcategories $\mathcal{P}(\phi)\in\mathcal{D}$ for each $\phi\in\mathbb{R}$, satisfying the following axioms:

			\par
			(a) for all $\phi \in \mathbb{R}$, $\mathcal{P}(\phi+1)=\mathcal{P}(\phi)[1]$,
			
			\par
			(b) if $\phi_1>\phi_2$ and $A_j\in\mathcal{P}(\phi_j)$ then $Hom_{\mathcal{D}}(A_1,A_2)=0$,
			
			\par

			(c) for every $0\neq E\in\mathcal{D}$ there is a sequence of real numbers
			
			$$\phi_1>\phi_2>\cdots>\phi_m$$and a sequence of morphisms 
			
			$$0=E_0\xrightarrow{f_1}E_1\xrightarrow{f_2} \cdots \xrightarrow{f_m}E_m=E $$such that the cone of $f_j$ is in $\mathcal{P}(\phi_j)$ for all $j$.
			
		\end{definition}
		
		This definition of slicings can be viewed as a refinement of t-structures in triangulated categories.
		
		\begin{definition}\label{first WSC}
			A stability condition on $\mathcal{D}$ consists of a pair $(\mathcal{P},Z)$, where $\mathcal{P}$ is a slicing and $Z:K_0(D)\rightarrow\mathbb{C}$ is a group homomorphism such that the following conditions are satisfied:
			\par

			(a) If $0\neq E\in\mathcal{P}(\phi)$ then $Z(E)=m(E)exp(i\pi\phi)$ for some $m(E)\in \mathbb{R}_{> 0}$.
			
			\par 
			
			(b) (Support property) The central charge $Z$ factors as $K_0(\mathcal{D})\xrightarrow{v} \Lambda\xrightarrow{g} \mathbb{C}$, where $\Lambda$ is a finite rank lattice, $v$ is a group homomorphism and $g$ is a group homomorphism, and there exists a quadratic form $Q$ on $\Lambda\otimes\mathbb{R}$ such that $Q|_{ker(g)}$ is negative definite, and $Q(v(E))\geq 0$, for any object $E\in\mathcal{P}(\phi)$. 
			
		\end{definition}

		\begin{remark}

			If we only require $m(E)$ to be nonnegative in (a), then the pair $(\mathcal{P},Z)$ is called a weak stability condition (see e.g. \cite[Definition 14.1]{Stabilityconditionsinfamilies}). By \cite[Lemma 2.2]{bridgeland2008stability}, there is a $S^1$ action on the space of stability conditions. Specifically, for any element $e^{i\theta}\in S^1$, $e^{i\theta}\cdot (Z,\mathcal{P})=(Z',\mathcal{P}')$ by setting $Z'=e^{i\theta} Z$ and $\mathcal{P}'(\phi)=\mathcal{P}(\phi-\theta)$.
		\end{remark}
		
		There is an equivalent way of defining a stability condition, which will be more frequently used in this paper. Firstly, we need to define what is a (weak) stability function $Z$ on an abelian category $\mathcal{A}$.
		\begin{definition} \label{weak stability condition}
			Let $\mathcal{A}$ be an abelian category. We call a group homomorphism $Z:K_0(\mathcal{A})\rightarrow \mathbb{C}$ a weak stability function on $\mathcal{A}$ if, for $E\in \mathcal{A}$, we have $Im(Z(E))\geq0$,with $Im(Z(E))=0 \implies  Re(Z(E))\leq0$. 
			If moreover, for $E\neq 0,$ $Im(Z(E))=0\implies Re(Z(E))< 0$, we say that $Z$ is a stability function.
		\end{definition}
		
		\begin{definition}\label{slicing }
			A (weak) stability condition on $\mathcal{D}$ is a pair $\sigma=(\mathcal{A},Z)$ consisting of the heart of a bounded t-structure $\mathcal{A}\subset\mathcal{D}$ and a (weak) stability function $Z:K_0(A)\rightarrow \mathbb{C}$ such that (a) and (b) below are satisfied:
			
			\par
			(a) (HN-filtration) The function $Z$ allows us to define a slope for any object $E$ in the heart $\mathcal{A}$ by
			
			$$\mu_{\sigma}(E):=\begin{cases} -\frac{Re(Z(E))}{Im(Z(E))}\ &\text{if} \  Im(Z(E))> 0,\\ +\infty &\text{otherwise.} \end{cases}$$

			The slope function gives a notion of stability: A nonzero object $E\in \mathcal{A}$ is $\sigma$ semistable  if for every proper subobject $F$, we have $\mu_{\sigma}(F)\leq  \mu_{\sigma}(E)$. Moreover, if  $\mu_{\sigma}(F)<  \mu_{\sigma}(E)$ holds for any proper subobject $F\subset E$, we say that $E$ is $\sigma$ stable.
			
			We require any nonzero object $E$ of $\mathcal{A}$ to have a Harder--Narasimhan filtration in $\sigma$ semistable ones, i.e., there exists a unique filtration $$0=E_0\subset E_1 \subset E_2\subset \cdots \subset E_{m-1} \subset E_m=E$$ such that $E_i/E_{i-1}$ is  $\sigma$ semistable and $\mu_{\sigma}(E_i/E_{i-1})>\mu_{\sigma}(E_{i+1}/E_i)$ for any $1\leq i\leq m$. The quotient $E_{i}/E_{i-1}$ is called the $i$-th HN factor of $E$.
			
			(b) (Support property) Equivalently as in Definition \ref{first WSC}, the central charge $Z$ factors as $K_0(\mathcal{D})\xrightarrow{v} \Lambda\xrightarrow{g} \mathbb{C}$. And there exists a quadratic form $Q$ on $\Lambda_{\mathbb{R}}$ such that $Q|_{ker(g)}$ is negative definite and $Q(v(E))\geq 0$ for any $\sigma$ semistable object $E\in\mathcal{A}$.
		\end{definition}

		\begin{remark}\label{abuse of terminology}
			  The equivalence of these two definitions is given by setting $\mathcal{P}(\phi)$ to be the full subcategory consisting of $\sigma$ semistable objects in $\mathcal{A}$. 
			
			Sometimes in this paper (especially in Section 5), a pair $(\mathcal{A},Z)$ of an abelian category $\mathcal{A}$ and a (weak) stability function $Z$ which only satisfies (a) will also be called a (weak) stability condition.
			
		\end{remark}
		 Although the following  result is well known, the author could not find a reference including the proof in full generality. So we include a proof here. 
		\begin{prop}
			If we have a pair $(\mathcal{A},Z)$, where $\mathcal{A}$ is an abelian category, and $Z:K_0(\mathcal{A})\rightarrow\mathbb{C}$ is a weak stability function. Assume that \begin{itemize}
				\item $\mathcal{A}$ is Noetherian,
				\item the image of $Im(Z)$ is discrete in $\mathbb{R}$.
				
			\end{itemize}
		Then Harder--Narasimhan filtrations exist in $\mathcal{A}$ with respect to $Z$.
		\end{prop}
	\begin{proof}
		The proof is essentially the same as the proof of \cite[Proposition 4.10]{macri2017lectures}. The only difference is that we have to deal with the more general case when $Z$ is a weak stability function.
		
		Let $E\in\mathcal{A}$ be a nonzero object. By Noetherianity of $\mathcal{A}$ and Lemma \ref{subabelian}, one can find a maximal subobject $E_0\subset E$ among the subobjects $\{F\subset E|Z(F)=0\}$. If $E_0=E$, the statement is trivial as $E$ is semistable. Assume that the quotient object $Q\coloneqq E/E_0\in\mathcal{A}$ is nonzero. By $\mathcal{H}'(Q)$ we denote the convex hull of the set of all $Z(F)$ for $F\subset Q$. By \cite[Lemma 4.9]{macri2017lectures} (one can easily go through the proof to see that this lemma also holds even if $Z$ is just a weak stability condition), this set is bounded from the left. By $\mathcal{H}_l$ we denote the half plane to the left of the line between $Z(Q)$ and $0$. If $Q$ is semistable,  the short exact sequence $$0\rightarrow E_0\rightarrow E\rightarrow Q\rightarrow 0$$ provides us the Harder--Narasimhan filtration of $E$ when $\mu(Q)<\infty$, and $E$ itself is semistable when $\mu(Q)=\infty$. Otherwise, the set $\mathcal{H}(Q)=\mathcal{H}'(Q)\cap \mathcal{H}_l$ is a convex polygon. 
	
	Let $v_0=0,v_1,\cdots, v_n=Z(Q)$ be the extremal vertices of $\mathcal{H}(Q)$, we claim that there exist subobjects $Q_i\subset Q$ with $Z(Q_i)=v_i$ for $i=1,\cdots, n-1$. Indeed, by definition of $\mathcal{H}(Q)$ there is  a sequence of objects $Q_{j,i}\subset Q$ such that $$\varinjlim_{j\to\infty} Z(Q_{j,i})=v_i.$$ Since $Im(Z)$ has discrete image, we can assume $Im(Z(Q_{j,i}))=Im(v_i)$ for all $j$. For any $\epsilon>0$, we can find $N\gg 0$ such that $|Z(Q_{j,i})-v_i|<\epsilon$ for any $j\geq N$. Consider the following short exact sequence $$0\rightarrow Q_{j_1,i}\cap Q_{j_2,i}\rightarrow  Q_{j_1,i} \oplus  Q_{j_2,i}  \rightarrow Q_{j_1,i}+ Q_{j_2,i}\rightarrow 0,$$ where $Q_{j_1,i}\cap Q_{j_2,i}$ and $Q_{j_1,i}+ Q_{j_2,i}$ are the intersection and sum of $Q_{j_1,i}, Q_{j_2,i}$  inside $Q$.  By the definition of weak stability functions, we have $Im(Z(Q_{j_1,i}+ Q_{j_2,i}))\geq Im(v_i)$ as $Q_{j_1},Q_{j_2}$ are subobjects of $Q_{j_1,i}+ Q_{j_2,i}$. Assume that $Im(Z(Q_{j_1,i}+ Q_{j_2,i}))>Im(v_i)$ for contrary,  we get that $Im(Z(Q_{j_1,i}\cap Q_{j_2,i}))<Im(v_i)$. If we take $\epsilon$ to be small enough, the segment connecting $Z(Q_{j_1,i}\cap Q_{j_2,i})$ and $Z(Q_{j_1,i}+ Q_{j_2,i})$ would be very close to the point $v_i$, and one of the points $Z(Q_{j_1,i}\cap Q_{j_2,i})$, $Z(Q_{j_1,i}+ Q_{j_2,i})$ would lie outside the convex hull $\mathcal{H}'(Q)$, contradicting the definition  of $\mathcal{H}'(Q)$.

	Hence, we can assume that $Im(Z(Q_{j_1,i}+ Q_{j_2,i}))=Im(v_i)$ for all $j_1, j_2\gg 0$. Moreover, as $Q_{j_1,i},Q_{j_2,i}$ are subobjects of $Q_{j_1,i}+Q_{j_2,i}$, we have that $$Re(v_i)\leq Re(Z(Q_{j_1,i}+Q_{j_2,i})\leq min\{Re(Z(Q_{j_1,i})), Re(Z(Q_{j_2,i}))\}$$ by Definition \ref{weak stability condition}. For a given $N\gg 0$, we  replace $Q_{j,i}$ by $Q_{j,i}+\cdots +Q_{N+1,i}+Q_{N,i}$ for any $j>N$. After this replacement, we can assume that $Q_{N,i}\subset Q_{N+1,i}\subset \cdots \subset Q$, and $$\varinjlim_{j\to\infty} Z(Q_{j,i})=v_i.$$ By Noetherianity of $\mathcal{A}$, we get a subobject $Q_i\subset Q$ with $Z(Q_i)=v_i$. Hence, the claim is proved. 
	
	Next, we will modify $Q_i$ a little bit so that the following conditions hold. \begin{enumerate}
		\item The inclusion $Q_{i-1}\subset Q_i$ holds for all $i=1,\cdots,n$.
		\item The object $G_i\coloneqq Q_i/Q_{i-1}$ is semistable for all $i=1,\cdots, n$.
	\end{enumerate}
Same as the proof of \cite[Proposition 4.10]{macri2017lectures}, one can show that $Z(Q_{i-1}\cap Q_i)=Z(Q_{i-1})$ and $Z(Q_{i-1}+Q_i)=Z(Q_i)$. Hence, we can replace $Q_i$ by $Q_1+\cdots +Q_i$ and get the condition (1). 

Similar to the proof of \cite[Proposition 4.10]{macri2017lectures}, one can show that the quotient object $G_i$ is quasi-semistable (see Definition \ref{abelianizer}) for all $i=1,\cdots n$, i.e., if we take $K_i$ to be the maximal subobject of $G_i$ among the sobobjects with trivial central charge, the quotient object $G_i/K_i$ will be semistable. We can modify $Q_i$ in the following way. We take $\widetilde{Q}_i$ to be the pullback of $K_{i+1}$ along the surjection $Q_{i+1}\rightarrow G_{i+1}$, as in the following commutative diagram. \[\begin{tikzcd}
		0 \arrow[r] & Q_i \arrow [r] \arrow[d,"id"]& \widetilde{Q}_i \arrow[r] \arrow[d]& K_{i+1}\arrow[r]\arrow[d] & 0 \\ 0\arrow[r]& Q_i\arrow[r] & Q_{i+1} \arrow[r]\arrow[d] &G_{i+1} \arrow[r] \arrow[d]& 0 \\ & & G_{i+1}/K_{i+1}\arrow[r, "id"] & G_{i+1}/K_{i+1} & 
	\end{tikzcd}\]
Each row and column in this diagram is a short exact sequence. If we modify $$Q_{n-1}, Q_{n-2},\cdots Q_1$$ inductively using this construction, we get a sequence of $Q_1\subset Q_2\subset\cdots \subset Q_{n-1}\subset Q$ with condition (2) holds. Hence, if we take $E_i$ to be the pullback of $Q_i$ along the surjection $E\rightarrow Q$ for any $1\leq i\leq n$, the sequence $E_0\subset E_1 \subset \cdots\subset E_{n-1}\subset  E$ is a Harder--Narasimhan filtration of $E$ if $\mu(E_1)<+\infty$. In the case $\mu(E_1)=+\infty$, the sequence $E_1 \subset \cdots\subset E_{n-1}\subset  E$ is a Harder--Narasimhan filtration of $E$. 
	\end{proof}
	\begin{remark}
		Although we do not ask for the uniqueness of Harder--Narasimhan filtration in this proposition, the uniqueness follows automatically from the definition of semistable objects. In fact, the proof of \cite[Lemma 1.3.7]{HarderNarasimhanfiltration} can be easily generalized to our case.
	\end{remark}
		There is an important operation called tilting with respect to a torsion pair, which is very useful for constructing stability conditions.
		\begin{definition}
			A torsion pair in an abelian category $\mathcal{A}$ is a pair of full subcategories $(\mathcal{T},\mathcal{F})$ of $\mathcal{A}$ which satisfy $\mathrm{Hom}_{\mathcal{A}}(T,F)=0$ for $T\in\mathcal{T}$ and $F\in\mathcal{F}$, and such that every object $E\in\mathcal{A}$ fits into a short exact sequence $$0\rightarrow T\rightarrow E\rightarrow F\rightarrow 0$$
			for some pair of objects $T\in\mathcal{T}$ and $F\in\mathcal{F}$.
		\end{definition}
		\begin{remark}\label{Torsion pair remark}
			In this paper, most torsion pairs are coming from weak stability conditions $\sigma=(\mathcal{A},Z)$. In fact, let $$\mathcal{T}=\{E\in\mathcal{A}\mid \mu_{\sigma,min}(E)>0\}\ and \ \mathcal{F}=\{E\in\mathcal{A}\mid \mu_{\sigma,max}(E)\leq 0\}$$ be a pair of full subcategories, where $\mu_{\sigma,min}(E)$ is the slope of the last HN-factor of $E$ and $\mu_{\sigma,max}(E)$ is the slope of the first HN-factor of $E$. It is easy to see that this is a torsion pair.
		\end{remark}
		\begin{lemma}[{\cite[Proposition 2.1]{happel1996tilting}}]\label{Tilt of torsion pairs}
			Suppose $\mathcal{A}$ is the heart of a bounded t-structure on a triangulated category $\mathcal{D}$, $(\mathcal{T},\mathcal{F})$ is a torsion pair in $\mathcal{A}$. Then $\mathcal{A}^{\#}=\langle\mathcal{F}[1], \mathcal{T}\rangle$ is a heart of a bounded t-structure on $\mathcal{D}$.
		\end{lemma}
		
		In this paper, we are interested in the case when $\mathcal{D}$ is the bounded derived category of coherent sheaves on an algebraic variety $X$ or an admissible component of it. From now on, $X$ will be a smooth projective variety over an algebraically closed field $k$, and $D(X)$ will be the bounded derived category of coherent sheaves on $X$.

	\section{AP heart and complexified Hilbert polynomial}\label{Preliminary results}
	In this section, we will recall some results from \cite{APsheaves}, \cite{polishchuk2007constant}, \cite{bayer2014mmp} and \cite{Stabilityconditionsonproductvarieties}. We will work under the following setup in the rest of this paper. 
	
	\textbf{Setup}: Suppose $X$ and $S$ are smooth projective varieties, and $\sigma=(\mathcal{A},Z)$ is a stability condition on $D(X)$, where $\mathcal{A}$ is Noetherian and the image of $Z$ is discrete. The heart $\mathcal{A}$ corresponds to a bounded $t$-structure $(D^{\leq 0}(X), D^{\geq 0}(X))$ on $D(X)$.

	For any t-structure  $(D^{\leq 0}(X), D^{\geq 0}(X))$ on $D(X)$, we have the following theorem. 
	
	\begin{theorem}[{\cite[Theorem 3.3.6]{polishchuk2007constant}}]
		
		Suppose $S$ is a projective variety of dimension $r$, and $\mathcal{O}(1)$ is an ample line bundle on $S$. There exists a global t-structure on $D(X\times S)$ defined as 
		$$D^{[a,b]}(X\times S)=\{E\in D(X\times S)\mid \textbf{R}p_*(E \otimes q^*(\mathcal{O}(n)))\in D^{[a,b]}(X) \ for\ all \ n\gg 0\}.$$
		
		Here $p,q$ are the projections from $X\times S$ to $X$ and $S$ respectively,  $a<b$ are integers  and can be infinite. Moreover, the global heart $$\mathcal{A}_S=D^{\leq 0}(X\times S)\cap D^{\geq 0}(X\times S)$$ is Noetherian and independent of the choice of $\mathcal{O}(1)$.

	\end{theorem} 
	\begin{remark}
		This theorem can be viewed as a generalization of Serre's vanishing. In fact, if $\mathcal{A}=coh(X)$ is the abelian category of coherent sheaves on $X$, then $\mathcal{A}_S=coh(X\times S)$ is the abelian category of coherent sheaves on $X\times S$.
	\end{remark}
\begin{remark}\label{Complexified Hilber polynomials}
		In \cite{Stabilityconditionsonproductvarieties}, we observed that $Z(\textbf{R}p_*(E\otimes q^*\mathcal{O}(n)))$ is a polynomial of degree no more than $dim(S)=r$, whose leading coefficient is a weak stability function on $\mathcal{A}_S$. We denote this polynomial by $L_E(n)$ for any object $E$ in $\mathcal{A}_S$.
\end{remark}  We also have the following theorem.
	
	\begin{theorem}[{\cite[Theorem 3.3]{Stabilityconditionsonproductvarieties}}]\label{glabal weak stablity condition}
		Assume $S$ is a smooth projective variety of dimension $r$, we define $(\mathcal{A}_S, Z_S)$ as below:
		\\
		$$\mathcal{A}_S=\{E\in D(X\times S)\mid \textbf{R}p_*(E \otimes q^*(\mathcal{O}(n)))\in \mathcal{A}\  for\ all\ n\gg0 \}$$
		$$Z_S(E)=\lim_{n\rightarrow +\infty}\frac{Z(\textbf{R}p_*(E\otimes q^*(\mathcal{O}(n)))r!}{n^r vol{(\mathcal{O}(1))}},$$where $vol(\mathcal{O}(1))$ is the volume of $\mathcal{O}(1)$ (see \cite[Definition 2.2.31]{lazarsfeld2004positivity} for the definition of volume of  line bundle). Then this pair is a weak stability condition on $D(X\times S)$.
	\end{theorem}
	
	\begin{example}
		If we take $X=Spec(\mathbb{C})$, $\mathcal{A}$ is the category of $\mathbb{C}$-vector spaces and $Z(V)=z\cdot dim(V)$ for any finite dimensional $\mathbb{C}$ vector space, where $z\in\mathbb{H}\cup\mathbb{R}_{\leq 0}$. Then the global heart is the category of coherent sheaves on $S$, and $L_E(n)=z\cdot Hilb_E(n)\ for\ n\gg 0$. Therefore, we call $L_E(n)$ the complexified Hilbert polynomial.
	\end{example}
	
	We get the following slope function $\mu_1$ from the pair $(\mathcal{A}_S,Z_S)$:
	
	$$\mu_1(E)\coloneqq\begin{cases} -\frac{Re(Z_S(E))}{Im(Z_S(E))} &\text{if} \ Im(Z_S(E))> 0,\\ +\infty &\text{otherwise.}\end{cases}$$
	
	This weak stability condition is closely related to the global slicing constructed in \cite[Section 4]{bayer2014mmp}. We include the explicit construction of the global slicing in the following paragraph. 
	
	Given a stability condition $\sigma=(Z, \mathcal{P})$ on $D(X)$ and a phase $\phi\in\mathbb{R}$, we have its associated t-structure $$(\mathcal{P}(>\phi)=\mathcal{D}_{\phi}^{\leq -1}(X) ,\mathcal{P}(\leq\phi)=\mathcal{D}_{\phi}^{\geq 0}(X))$$ on $D(X)$. By Abramovich and Polishchuk's construction, we get a global t-structure $$(\mathcal{P}_S(>\phi)\coloneqq\mathcal{D}_{\phi}^{\leq -1}(X\times S), \mathcal{P}_S(\leq \phi)\coloneqq\mathcal{D}_{\phi}^{\geq 0}(X\times S))$$ on $D(X\times S)$. Then we have the following lemma in \cite[Section 4]{bayer2014mmp}.
	
	\begin{lemma}[{\cite[Lemma 4.6]{bayer2014mmp}}]\label{global slicing}
		Assume $\sigma=(Z, \mathcal{P})$ is a weak stability condition as in our setup, and $\mathcal{P}_S(>\phi)$, $\mathcal{P}_S(\leq \phi)$ defined as above. There is a slicing $\mathcal{P}_S$ on $D^b(X\times S)$ defined by
		
		$$\mathcal{P}_S(\phi)=\mathcal{P}_S(\leq \phi)\cap \underset{\epsilon>0}{\cap}\mathcal{P}_S(>\phi-\epsilon).$$
	\end{lemma}

By Lemma \ref{global slicing}, one can describe the relation between $\mu_1$ semistable objects and the global slicing $\mathcal{P}_S$. 
	
	\begin{prop}[{\cite[Proposition 3.14]{Stabilityconditionsonproductvarieties}}]\label{semistable reduction}
		If  $S$ is a smooth projective variety, and $E\in \mathcal{A}_S$ is semistable with respect to $\mu_1$ of phase $\phi$ and $Z_S(E)\neq 0$, then there exists a short exact sequence 
		
		$$0\rightarrow K\rightarrow E\rightarrow Q\rightarrow 0$$
		such that $K\in \mathcal{P}_S(\phi)$, $Q\in \mathcal{P}_S(<\phi)$ and $Z_S(Q)=0$.
	\end{prop}
	
	\section{Abelian subcategories in $\mathcal{A}_S$}\label{Abelian subcategories}
		For the simplicity of our statements and arguments, we introduce the following definition.
	\begin{definition}\label{rational stability condition}
		If $(\mathcal{A},Z)$ is a stability condition, and the image of $Z$ lies in $\mathbb{Q}\oplus \mathbb{Q}i$, we call $(\mathcal{A},Z)$ a rational stability condition. We use $RStab(X)$ to denote the set of rational stability conditions on $D(X)$.
	\end{definition}
	\begin{remark}
		By \cite[Proposition 5.0.1]{APsheaves}, we know the heart $\mathcal{A}$ of a rational stability condition is Noetherian. We focus on the rational stability conditions just for the simplicity of statements and arguments. All results and proofs in the rest of this paper can be easily adapted to the stability conditions whose heart is Noetherian and central charge has discrete image in its imaginary part. 
	\end{remark}
	From now on,  unless otherwise specified, we  assume that our original stability condition $\sigma$ on $X$ is a rational stability condition. As in Remark \ref{Complexified Hilber polynomials}, for any $E\in\mathcal{A}_S$, we can write $$L_E(n)\coloneqq\Sigma_{k=0}^{r} (a_k(E)+ib_k(E))n^k,$$ where $a_i, b_i:K_0(\mathcal{A}_S)\rightarrow\mathbb{Q}$ are group homomorphisms for $0\leq i\leq r$.
	
	Then we have the following positivity of coefficients $a_k$ and $b_k$.
	
	\begin{prop} \label{positive coefficients}
		For any $E\in\mathcal{A}_S$, we have the following inequalities.
		
		(1) $b_r(E)\geq 0$.
		
		(2) If $b_r(E)=0$, then $a_r(E)\leq0$ and $b_{r-1}(E)\geq 0$.

		(3) In general, if $$b_r(E)=a_r(E)=b_{r-1}(E)=\cdots=a_i(E)=b_{i-1}(E)=0,$$ then $a_{i-1}(E)\leq 0$ and $b_{i-2}(E)\geq 0$ for any $2\leq i\leq r$.
		
		(4) Moreover, if $E$ is a nonzero object and $$b_r(E)=a_r(E)=b_{r-1}(E)=\cdots=a_1(E)=b_{0}(E)=0,$$ then $a_0(E)<0$.
		
	\end{prop}
	\begin{proof}
		See the proof in \cite[Lemma 4.1]{somequadraticinequalities}.
	\end{proof}
	
	There are some full subcategories of $\mathcal{A}_S$.
	\begin{definition}
		For any integer $1\leq j\leq r$, we can define a full subcategory $\mathcal{A}_S^{\leq j} $ in $\mathcal{A}_S$ in the following way: $$obj(\mathcal{A}_S^{\leq j})=\{E\in\mathcal{A}_S\ | deg(L_E(n))\leq j\}.$$
	\end{definition}
	
	\begin{lemma}\label{subquotients}
		If $E$ is an object in $\mathcal{A}_S^{\leq j}$, and $Q$ is a subquotient of $E$ in $\mathcal{A}_S$, then $Q\in\mathcal{A}_S^{\leq j}$.
	\end{lemma}
	\begin{proof}
		Assume $E$ is an object in $\mathcal{A}_S^{\leq j}$, let $$0\rightarrow K\rightarrow E\rightarrow Q\rightarrow 0$$ be a short exact sequence in $\mathcal{A}_S$.
		
		Suppose $k,q$ are the maximal integers such that $a_k(K)+ib_k(K)\neq 0$ and $a_q(Q)+ib_q(Q)\neq 0$ respectively. 
		If $k>j$, then we get $$a_{k}(K)+ib_k(K)=-a_k(Q)-ib_k(Q).$$
		
		By Proposition \ref{positive coefficients}, we get $b_k(K)\geq 0$, and if $b_k(K)=0$, we have $a_k(K)<0$ by the definition of $k$. Then $b_k(Q)\leq 0$, and if $b_k(Q)=0$, we have $a_k(Q)>0$. Therefore, we have $q>k$ by Proposition \ref{positive coefficients}. Similarly, we can prove that $k>q$, a contradiction.
		
		Therefore, $k,q$ can not exceed $j$, which implies $K,Q\in\mathcal{A}_S^{\leq j}$.
	\end{proof}
	\begin{corollary}\label{abelian}
		$\mathcal{A}_S^{\leq j}$ is an abelian category for any $0\leq j\leq r$.
	\end{corollary}

	\begin{prop}\label{support torsion pair}
		There exists subcategory $\mathcal{F}^{>j}$ in $\mathcal{A}_S$ for any $1\leq j\leq r$, such that the pair $(\mathcal{A}_S^{\leq j}, \mathcal{F}^{>j})$ is a torsion pair of $\mathcal{A}_S$.
	\end{prop}
	\begin{proof}
		We claim that for any object $E\in\mathcal{A}_S$, there exists a maximal subobject $E_j \in\mathcal{A}_S^{\leq j}$ of $E$.
		
		To prove the claim, we firstly show that for any two subobjects $E_1,E_2\subset E$ and $E_1,E_2\in\mathcal{A}_S^{\leq j}$, there exists a third subobject $E_3\in\mathcal{A}_S^{\leq j}$ such that $E_1, E_2\subset E_3\subset E$.
		
		Suppose we have two injections $$E_1\lhook\joinrel\xrightarrow{f} E\  \ and\ \  E_2\lhook\joinrel\xrightarrow{g} E.$$

		Then we have a morphism $$ E_1\oplus E_2\xrightarrow{(f,g)}E.$$
		
		Let $E_3$ be the image of this morphism, by Lemma \ref{subquotients}, we know that $E_3\in\mathcal{A}_S^{\leq j}$. And $E_1,E_2\subset E_3\subset E$ by construction.
		
		Hence it suffices to prove that the increasing sequence $$E_1\subset E_2\subset\cdots\subset E$$ of subobjects $E_k\in\mathcal{A}_S^{\leq j}$ stabilizes after finite step. This follows directly from the fact $\mathcal{A}_S$ is Noetherian.
		
		Let $F^{>j}$ be the full subcategory consisting of the objects whose maximal subobjects in $\mathcal{A}_S^{\leq j}$ is $0$. Then it is easy to see that $(\mathcal{A}_S^{\leq j}, \mathcal{F}^{>j})$ is a torsion pair.
		
	\end{proof}

	\begin{corollary}
		There exists a unique filtration $$0=E_0\subset E_1\subset\cdots\subset E_r=E$$ for any object $E\in\mathcal{A}_S$, such that $E_j$ is the maximal subobject  of $E$ in $\mathcal{A}_S^{\leq j}$.
	\end{corollary}
	\begin{proof}
		By the proof of Proposition \ref{support torsion pair}, we can take $E_j$ to be the maximal subobject in $\mathcal{A}_S^{\leq j}$. Since $\mathcal{A}_S^{\leq j}$ is a full subcategory of $\mathcal{A}_S^{\leq k}$ for any integers $1\leq j<k\leq r$, we get $E_j\subset E_k$, hence we get the unique filtration.
	\end{proof}

	\begin{remark}
		One can view this as a generalization of the torsion filtration of a sheaf (see \cite[Definition 1.1.4]{huybrechts2010geometry}).
	\end{remark}

		If we denote the minimal triangulated subcategory generated by $\mathcal{A}_S^{\leq 1}$ in $D(X\times S)$ by $\mathcal{D}^{\leq 1}_S$, then we can apply the construction in \cite{Stabilityconditionsonproductvarieties} on $\mathcal{D}_S^{\leq 1}$.
		
		\begin{corollary}
			The triangulated category $\mathcal{D}_S^{\leq 1}$ admits stability conditions.
		\end{corollary}
\begin{proof}
By Proposition \ref{positive coefficients}, we know that  for any $t\in\mathbb{Q}_{>0}$, the central charge $Z_t$ defined by $$Z_t(E)=a_1(E)t-b_0(E)+i\cdot b_1(E)t$$ is a weak stability function on $\mathcal{A}_S^{\leq 1}$.  And the pair $\sigma_t\coloneqq (\mathcal{A}_S^{\leq 1}, Z_t)$ is a weak stability condition by the rationality of the image of $Z_t$ and Noetherianity of $\mathcal{A}_S^{\leq1}$. Hence this gives us a torsion pair as in Remark \ref{Torsion pair remark}, and we tilt the heart $\mathcal{A}_S^{\leq 1}$ in $\mathcal{D}_S^{\leq 1}$ with respect to such a torsion pair as in Lemma \ref{Tilt of torsion pairs}. We use $\mathcal{A}_S^{\leq 1,t}$ to denote the tilted heart. 

By \cite[Theorem 4.5]{somequadraticinequalities}, we know that if $E\in \mathcal{A}_S^{\leq 1}$ is $\sigma_t$ semistable, we have that $$b_1(E)a_0(E)-b_0(E)a_1(E)\geq 0.$$

Therefore, we can use the results in \cite{Stabilityconditionsonproductvarieties} to show that the pair $(\mathcal{A}_S^{\leq 1,t}, Z_S^{s,t})$, where $Z_S^{s,t}(E)=b_1(E)s+a_0(E)+i(-a_1(E)t+b_0(E))$, is a stability condition on $\mathcal{D}^{\leq 1}$ for any $s,t\in\mathbb{Q}_{>0}$. We omit the proof as it is exact the same as in \cite[Sections 4-5]{Stabilityconditionsonproductvarieties}.
\end{proof}
	
	\subsection{Abelian subcategories from a weak stability condition}

	In this subsection, we consider a more general situation. Let $\sigma=(\mathcal{A},Z)$ be an arbitrary weak stability condition on a triangulated category $\mathcal{D}$.
	
	\begin{lemma}\label{subabelian}
		Let  $\mathcal{A}_0$ denote the full subcategory in $\mathcal{A}$ consisting of objects $E$ whose central charge is $0$, then $\mathcal{A}_0$ is an abelian subcatgory.
	\end{lemma}
	
	\begin{proof}
		It is very easy to see that $\mathcal{A}_0$ is closed under subobjects and quotient objects. In particular, $\mathcal{A}_0$ is an abelian subcategory.
	\end{proof}

	\begin{remark}
		Corollary \ref{abelian} can also be viewed as a direct consequence of Lemma \ref{subabelian}. Indeed, for any $0\leq j\leq r$, $(\mathcal{A}_S^{\leq j}, a_j+ib_j)$ is a weak stability condition, hence we can apply Lemma \ref{subabelian} inductively.
		
	\end{remark}

	Let $\mathcal{P}_{\sigma}(\phi)$ be the subcategory of $\sigma$ semistable objects whose phase is $\phi$. It is easy to see that $\mathcal{A}_0\subset\mathcal{P}_{\sigma}(1)$.

	\begin{lemma}\label{switching lemma}
		If we have a short exact sequence in $\mathcal{A}$
		$$0\rightarrow K\rightarrow E\xrightarrow{f} Q\rightarrow 0,$$ where $K\in\mathcal{P}_{\sigma}(\phi)$ and $Q\in\mathcal{A}_0$, then we have the following short exact sequence in $\mathcal{A}$
		
		$$0\rightarrow Q'\rightarrow E\rightarrow K'\rightarrow 0,$$ where $Q'\in\mathcal{A}_0$ and $K'\in\mathcal{P}_{\sigma}(\phi)$.
	\end{lemma}
	
	\begin{proof}
		If $K\in\mathcal{P}_{\sigma}(1)$, we know that $E$ is in $\mathcal{P}_{\sigma}(1)$ as well. Hence we can take $K'=E$ and $Q'=0$.
		
		Now we can assume that $0<\phi<1$. By the HN property of $\sigma$, there exists a maximal subobject $Q'\in\mathcal{P}_{\sigma}(1)$ with a natural inclusion $i: Q'\hookrightarrow E$. Hence we have the following short exact sequence
		$$0\rightarrow Q'\xrightarrow{i} E\xrightarrow{g} K'\rightarrow 0.$$
		We claim that $Z(Q')=0$ and $K'\in\mathcal{P}_{\sigma}(\phi)$. 
		
		If $Z(Q')\neq 0$, then we have $Z(Q')\in\mathbb{R}_{<0}$. Let $G\coloneqq ker(f\circ i):Q'\rightarrow Q$, then we have $Z(G)=Z(Q')\in\mathbb{R}_{<0}$ and $G\hookrightarrow K$. This contradicts the facts $K\in\mathcal{P}_{\sigma}(\phi)$ and $\phi<1$. Hence, we get $Z(Q')=0$.
		
		Now we know that $Z(E)=Z(K')$. Suppose that $K'$ is not semistable. Let $F'$ be a destabilizing subobject of $K'$ and $F=g^{-1}(F')$. Then we have $Z(F)=Z(F')$ as $Z(Q')=0$ and a morphism $f':F\rightarrow Q$ which is the composition of $f$ and the inclusion from $F$ to $E$.  
		
		Let $H\coloneqq ker(f')$. We know that $Z(H)=Z(F)$ and $H$ is a subobject of $K$. Since $Z(K)=Z(K')=Z(E)$, $H$ destabilizes $K$ as $\mu(H)=\mu(F)=\mu(F')<\mu(K')=\mu(K)$, which contradicts the assumption that $K$ is semistable.

		Therefore, the claim is proved, and we get the short exact sequence.
	\end{proof}
	
	\begin{remark}
		In the proof of this lemma, we also proved that $Q'$ is a subobject of $Q$.
	\end{remark}
	
	\begin{definition}\label{definition of abelianizer}
		We call $\mathcal{A}_0$ the abelianizer of the weak stability condition $\sigma=(\mathcal{A},Z)$. If there exists a short exact sequence $$0\rightarrow Q'\rightarrow E\rightarrow K'\rightarrow 0,$$ where $Q'\in\mathcal{A}_0$ and $K'\in\mathcal{P}_{\sigma}(\phi)$, we say that $E$ is quasi-semistable.
	\end{definition}

	The following Proposition is the justification of naming $\mathcal{A}_0$ the abelianizer of $\sigma$.

	\begin{prop}\label{abelianizer}
		Let $$\mathcal{A}_{\sigma}(\phi)\coloneqq\{E\in\mathcal{A}\ |E\ is\ quasi- semistable \ and \ Z(E)\in exp(i\pi \phi)\cdot \mathbb{R}_{\geq 0}\}.$$ Then $\mathcal{A}_{\sigma}(\phi)$ is an abelian subcategory.
	\end{prop}
	
	\begin{proof}
		If $\phi=1$, then $\mathcal{A}_{\sigma}(1)=\mathcal{P}_{\sigma}(1)$. Then $\mathcal{A}_{\sigma}(1)$ is an abelian category since it is closed under subobjects and quotient objects.

		Now we can assume $0<\phi<1$. If we have a morphism $f:E_1\rightarrow E_2$ in $\mathcal{A}$ between two objects $E_1,E_2\in\mathcal{A}_{\sigma}(\phi)$. There exists the following diagram in $\mathcal{A}$.
		
		\[ \begin{tikzcd}
		0 \arrow{r}{}  & Q_1  \arrow{r}{} & E_1 \arrow{r}{} \arrow{d}[swap]{f} & K_1 \arrow{r} & 0 \\%
		0 \arrow{r}{}& Q_2 \arrow{r}{}  & E_2 \arrow{r}{} & K_2 \arrow{r}{} & 0
		\end{tikzcd}
		\]
		where $Q_i\in\mathcal{A}_0$ and $K_i\in\mathcal{P}_{\sigma}(\phi)$ for $i=1,2$.		
		
		Since $0<\phi<1$ and $Q_i\in\mathcal{A}_0$, there is no nontrivial morphism from $Q_1$ to $K_2$. Hence the diagram can be completed by the following commutative diagram.
		
		\[ \begin{tikzcd}
		0 \arrow{r}{}  & Q_1  \arrow{r}{} \arrow{d}[swap]{h}& E_1 \arrow{r}{} \arrow{d}[swap]{f} & K_1 \arrow{r} \arrow{d}[swap]{g}& 0 \\%
		0 \arrow{r}{}& Q_2 \arrow{r}{}  & E_2 \arrow{r}{} & K_2 \arrow{r}{} & 0.
		\end{tikzcd}
		\]
		It suffices to show that $ker(f), im(f)$ and $coker(f)$ are quasi-semistable.
		
		By snake lemma, we have the following long exact sequence
		
		$$0\rightarrow ker(h)\rightarrow ker(f)\rightarrow ker(g)\xrightarrow{\delta}coker(h)\rightarrow coker(f)\rightarrow coker(g)\rightarrow 0.$$
		
		This can be decomposed into two short exact sequences
		$$0\rightarrow ker(h)\rightarrow ker(f)\rightarrow ker(\delta)\rightarrow 0$$and $$0\rightarrow coker(\delta)\rightarrow coker(f)\rightarrow coker(g)\rightarrow 0.$$
		
		We know that  $im(g)$ is a subobject of $K_2$ and a quotient object of $K_1$, and $K_1,K_2\in\mathcal{P}_{\sigma}(\phi)$. This implies that $im(g)\in\mathcal{P}_{\sigma}(\phi)$, hence $ker(g)\in\mathcal{P}_{\sigma}(\phi)$ as well. Since $coker(h)$ is a quotient object of $Q_2$, it is an object in $\mathcal{A}_0$ by Lemma \ref{subabelian}. Therefore, $Z(ker(\delta))=Z(ker(g))$ and $ker(\delta)\subset ker(g)$ imply that $ker(\delta)\in\mathcal{P}_{\sigma}(\phi)$. Thus we get that $ker(f)$ is quasi-semistable of phase $\phi$.
		
		Now we want to show that $coker(f)\in\mathcal{A}_{\sigma}(\phi)$. It is easy to see that $coker(g)\in\mathcal{A}_{\sigma}(\phi)$. Indeed, suppose $F$ is a subobject of $coker(g)$ that destabilizes $coker(g)$. Let $H\coloneqq coker(g)/F$, then if $Z(F)\neq 0$, we have $Z(H)\neq 0$ and $\mu_{\sigma}(H)<\mu_{\sigma}(coker(g))=\mu_{\sigma}(K_2)$. This contradicts to the assumption $K_2\in\mathcal{P}_{\sigma}(\phi)$ as $H$ is a quotient object of $K_2$. Hence we proved that any destabilizing subobject of $coker(g)$ is in $\mathcal{A}_0$. This is equivalent to $coker(g)\in\mathcal{A}_{\sigma}(\phi)$.
		
		By Definition \ref{definition of abelianizer}, we have $Q' \in\mathcal{A}_0$ and $K\in\mathcal{P}_{\sigma}(\phi)$ and  the following short exact sequence $$0\rightarrow Q'\rightarrow coker(g)\rightarrow K\rightarrow 0.$$Furthermore, we have the following commutative diagram by snake's lemma.
		
		\[ \begin{tikzcd}
		0 \arrow{r}{}  & coker(\delta)  \arrow{r}{} \arrow{d}[swap]{id}& Q \arrow{r}{} \arrow{d}[swap]{} & Q' \arrow{r} \arrow{d}[swap]{}& 0 \\%
		0 \arrow{r}{}& coker(\delta) \arrow{r}{}  \arrow{d}[swap]{}  & coker(f) \arrow{r}{}  \arrow{d}[swap]{}& coker(g) \arrow{r}{} \arrow{d}[swap]{}& 0 \\%
		0 \arrow{r} & 0 \arrow{r}{}& K\arrow{r}{id} & K\arrow{r}{} & 0	\end{tikzcd}
		\]
		where $Q$ is the kernel of the composed morphism $coker(f)\rightarrow coker(g)\rightarrow K$, and every row and every column is a short exact sequence. Notice that $coker(\delta)$ is a quotient object of $coker(h)\in\mathcal{A}_0$, hence $Q\in\mathcal{A}_0$, and $coker(f)\in\mathcal{A}_{\sigma}(\phi)$.
		
		It is left to prove that $im(f)\in\mathcal{A}_{\sigma}(\phi)$. By snake lemma, we have the following commutative diagram.
		\[ \begin{tikzcd}
		0 \arrow{r}{}  & ker(h)  \arrow{r}{} \arrow{d}[swap]{}& ker(f) \arrow{r}{} \arrow{d}[swap]{} & ker(\delta) \arrow{r} \arrow{d}[swap]{}& 0 \\%
		0 \arrow{r}{}& Q_1 \arrow{r}{}  \arrow{d}[swap]{}  & E_1 \arrow{r}{}  \arrow{d}[swap]{}& K_1 \arrow{r}{} \arrow{d}[swap]{}& 0 \\%
		0 \arrow{r} & im(h) \arrow{r}{}& im(f) \arrow{r}{} & J\arrow{r}{} & 0	\end{tikzcd}
		\]
		where $J\coloneqq K_1/ker(\delta)$, $im(h)\in\mathcal{A}_0$ and every column is a short exact sequence. Hence it suffices to prove that $J\in\mathcal{A}_{\sigma}(\phi)$. This follows from the following commutative diagram.
		\[ \begin{tikzcd}
		& & & 0\arrow{d}[swap]{}\\
		& 0 \arrow{r}{} \arrow{d}[swap]{} & 0 \arrow{r}{} \arrow{d}[swap]{} & im(\delta) \arrow{d}[swap]{}  & \\%
		0\arrow{r}{}   & ker(\delta)  \arrow{r}{} \arrow{d}[swap]{}& K_1 \arrow{r}{} \arrow{d}[swap]{id} & J \arrow{r} \arrow{d}[swap]{}& 0 \\
		0 \arrow{r}{}& ker(g) \arrow{r}{}  \arrow{d}[swap]{}  & K_1 \arrow{r}{}  \arrow{d}[swap]{}& im(g) \arrow{r}{} \arrow{d}[swap]{} & 0 \\%
		& im(\delta) \arrow{r}{}\arrow{d}[swap]{}& 0 \arrow{r}{} & 0 \\
		& 0 \end{tikzcd}
		\]
		where $im(\delta)\in\mathcal{A}_0$ and $im(g)\in\mathcal{P}_{\sigma}(\phi)$, and every column is a short exact sequence. Hence $J\in\mathcal{A}_{\sigma}(\phi)$, the proof is complete.
	\end{proof}
	
	\begin{remark}
		When $\sigma=(\mathcal{A},Z)$ is a stability condition, we know that $\mathcal{A}_0$ is trivial. In this case, Proposition \ref{abelianizer} is \cite[Lemma 5.2]{bridgeland2007stability}. 
	\end{remark}

	\section{Lexicographic order filtrations and torsion pairs}\label{torsion pairs}
	
	In this section, we further study the possible applications of Abramovich-Polishchuk's construction on stability conditions. More explicitly, we want to generalize the construction in \cite{Stabilityconditionsonproductvarieties} to the case when $S$ is a higher dimensional projective variety. The idea is to do multiple times of tilts on $\mathcal{A}_S$ by induction, and prove suitable inequalities at each inductive step. However, the usual tilted heart (in Remark \ref{Torsion pair remark}) does not satisfy the inequalities we need. In this section, we come up with a finer cut of $\mathcal{A}_S$ into torsion pairs (Corollary \ref{finer cuts}), whose tilted heart satisfies the inequalities we need (see Theorem \ref{positivity on tilted heart}). This leads us to the notion of $l$-th level semistable objects (Definition \ref{l-thlevel definition}), which is a generalization of the slope stability and the Gieseker stability. We expect this to be the first step towards the generalization of the construction in \cite{Stabilityconditionsonproductvarieties}.

	Let us focus on the abelian subcategories $\mathcal{A}_S^{\leq j}$, Proposition \ref{positive coefficients} and Proposition \ref{abelianizer} give us lots of weak stability conditions and abelian subcategories in $\mathcal{A}_S$.  As in Remark \ref{abuse of terminology}, we call a pair $(\mathcal{A},Z)$ a stability condition if this pair admits HN property (without the reference to any triangulated categories and t-structures).

	\begin{lemma}\label{lemma: first weak stability condition}
		Let $Z_{j}^t(E)=a_j(E)t-b_{j-1}(E)+ib_j(E)$, then $\sigma^t\coloneqq(\mathcal{A}_S^{\leq j}, Z^{t}_j)$ is a weak stability condition for any $t\in\mathbb{Q}_{> 0}$ and any integer $0 \leq j\leq r$.
	\end{lemma}
	\begin{proof}
		By Proposition \ref{positive coefficients},	$Z_j^t$ is a weak stability function on $\mathcal{A}_S^{\leq j}$. Then by the rationality of $b_j, a_j$ and the Noetherianity of $\mathcal{A}_S^{\leq j}$, we see that $\sigma^t$ admits HN property.
	\end{proof}

	Apply Proposition \ref{abelianizer} on the weak stability condition $\sigma^t$ in Lemma \ref{lemma: first weak stability condition}, we have the following abelian subcategories in $\mathcal{A}_S$.
	
	\begin{definition}
		For the simplicity of notation, let $\mathcal{A}^t(\phi)$ denote $\mathcal{A}_{\sigma^t}(\phi)$. Then $\mathcal{A}^t(\phi)$ is an abelian subcategory by Proposition \ref{abelianizer}.
	\end{definition}

	\begin{lemma}\label{first slice stability condition} We have the following stability conditions.

		(1) The pair $$\sigma_{1}^{t,t'}=(\mathcal{A}^t(1), a_{j-1}t'-b_{j-2}+i(b_{j-1}-a_jt))$$ is a weak stability condition for any $t,t'\in\mathbb{Q}_{> 0}$.
		
		(2) The pair $$\sigma_{\phi}^{t,t'}=(\mathcal{A}^t(\phi), a_{j-1}t'-b_{j-2}+ib_j)$$ is a weak stability condition for any $t, t'\in\mathbb{Q}_{> 0}$ and $\phi\in(0,1)$.
	\end{lemma}
	
	\begin{proof}
		For (1), assume that $E\in\mathcal{A}^t(1)$ and $b_{j-1}(E)-a_j(E)t=0$. Then we get $b_j(E)=a_j(E)=b_{j-1}(E)=0$. By Proposition \ref{positive coefficients}, this implies that $a_{j-1}(E)\leq 0$ and $b_{j-2}(E)\geq 0$. Therefore, we know that $a_{j-1}t'-b_{j-2}+ib_j$ is a weak stability function on $\mathcal{A}^t(\phi)$.

		For (2), assume that $E\in\mathcal{A}^t(\phi)$ and $b_j(E)=0$. Then since $\phi\in(0,1)$, we have that $Z_j^t(E)=0$, hence $b_j(E)=a_j(E)=b_{j-1}(E)=0$.  Therefore, by the same reason, we know that $a_{j-1}t'-b_{j-2}+ib_j$ is a weak stability function on $\mathcal{A}^t(\phi)$.
		
		The HN property following from rationality of central charge and the Noetherianity of $\mathcal{A}^t(\phi)$.
	\end{proof}

	We can construct weak stability conditions inductively by Proposition \ref{positive coefficients}. For convention, we define $b_{-1}=0$.

	\begin{lemma}\label{slice stability condition}

		(1) Inductively, for any $1\leq k\leq j$, let $\mathcal{A}^{t_1,t_2,\cdots,t_k}_{\phi_1,\phi_2,\cdots,\phi_{k}}\coloneqq\mathcal{A}_{\sigma_{\phi_1,\phi_2,\cdots,\phi_{k-1}}^{t_1,t_2,\cdots, t_k}}(\phi_k) $, then the pair $$\sigma_{\phi_1,\phi_2,\cdots,\phi_k}^{t_1,t_2,\cdots,t_{k+1}}=(\mathcal{A}^{t_1,t_2,\cdots,t_k}_{\phi_1,\phi_2\cdots,\phi_k}, a_{j-k}t_{k+1}-b_{j-k-1}+ib_j)$$ is a weak stability condition for any $t_1,\cdots, t_{k+1}\in\mathbb{Q}_{> 0}$ and any $\phi_1\cdots,\phi_k\in(0,1)$.
		
		(2) Given a sequence of real numbers $\phi_1,\cdots,\phi_k\in(0,1]$, let $M=\{1\leq m\leq k| \phi_i=1 \}$. Assume that $M$ is a nonempty set, and let $n$ be the maximal integer in $M$. Then the pair $$\sigma_{\phi_1,\phi_2,\cdots,\phi_k}^{t_1,t_2,\cdots,t_{k+1}}=(\mathcal{A}^{t_1,t_2,\cdots,t_k}_{\phi_1,\phi_2\cdots,\phi_k}, a_{j-k}t_{k+1}-b_{j-k-1}-i(a_{j-n+1}t_n-b_{j-n}))$$ is a weak stability condition for any $t_1,\cdots, t_{k+1}\in\mathbb{Q}_{> 0}$.
	\end{lemma}
	\begin{proof}
		
		Firstly, we claim that for any $t_1,t_2,\cdots,t_{k+1}\in\mathbb{Q}_{>0}$ and $\phi_1,\cdots,\phi_k\in(0,1]$, the abelianizer (see Definition \ref{definition of abelianizer}) of $\sigma_{\phi_1,\phi_2,\cdots,\phi_k}^{t_1,t_2,\cdots,t_{k+1}}$ is the full abelian subcategory consists of the following objects $$\{E\in\mathcal{A}_S^{\leq j}| b_j(E)=a_j(E)=\cdots=b_{j-k}(E)=a_{j-k}(E)=b_{j-k-1}(E)=0 \}.$$
		
		The claim and lemma can be proved by induction on $k$, the case $k=1$ follows from Lemma \ref{first slice stability condition} and Proposition \ref{positive coefficients}. 
		
		Let $(C_l)$ denote the proposition that the claim is true for all $1\leq k\leq l$, and $(L_l)$ denote the proposition that the lemma is true for all $1\leq k\leq l$. Then we will prove that $(C_{l-1})$ implies $(L_l)$ and $(L_l)$ implies $(C_l)$. Hence both propositions are true for any $1\leq k\leq j$.
		
		Indeed, assume that $(C_{k-1})$ is true. If $\phi_k<1$, we assume that  $E\in\mathcal{A}_{\phi_1,\cdots,\phi_k}^{t_1,\cdots,t_k}$ and the imaginary part of $E$ in $\sigma_{\phi_1,\phi_2,\cdots,\phi_k}^{t_1,t_2,\cdots,t_{k+1}}$ is $0$. Since $\phi_k<1$, the imaginary parts of $\sigma_{\phi_1,\phi_2,\cdots,\phi_k}^{t_1,t_2,\cdots,t_{k+1}}$ and $\sigma_{\phi_1,\phi_2,\cdots,\phi_{k-1}}^{t_1,t_2,\cdots,t_{k}}$ are the same by construction. By definition of $\mathcal{A}^{t_1,t_2,\cdots,t_k}_{\phi_1,\phi_2\cdots,\phi_k}$, we know that $E$ is in the abelianizer of $\sigma_{\phi_1,\phi_2,\cdots,\phi_{k-1}}^{t_1,t_2,\cdots,t_{k}}$. Hence by $(C_{k-1})$, we get that $$b_j(E)=a_j(E)=\cdots=b_{j-k+1}(E)=a_{j-k+1}(E)=b_{j-k}(E)=0.$$ Therefore, by Proposition \ref{positive coefficients}, $a_{j-k}(E)t_{k+1}-b_{j-k-1}(E)\leq 0$, hence $(L_k)$ and $(C_k)$ holds.
		
		Now we consider  the case when $\phi_k=1$. We assume that  $E\in\mathcal{A}_{\phi_1,\cdots,\phi_k}^{t_1,\cdots,t_k}$ and the imaginary part of $E$ in $\sigma_{\phi_1,\phi_2,\cdots,\phi_k}^{t_1,t_2,\cdots,t_{k+1}}$ is $0$, i.e., $a_{j-k+1}(E)t_k-b_{j-k}(E)=0$. Since $a_{j-k+1}t_k-b_{j-k}$ is the real part of $\sigma_{\phi_1,\phi_2,\cdots,\phi_{k-1}}^{t_1,t_2,\cdots,t_{k}}$ and $E$ is semistable of phase 1 with respect to $\sigma_{\phi_1,\phi_2,\cdots,\phi_{k-1}}^{t_1,t_2,\cdots,t_{k}}$. We get that $E$ is in the abelianizer of $\sigma_{\phi_1,\phi_2,\cdots,\phi_{k-1}}^{t_1,t_2,\cdots,t_{k}}$. Hence $(L_k)$ and $(C_k)$ holds.

		The HN property following from rationality of central charge and the Noetherianity of $\mathcal{A}^{t_1,t_2,\cdots,t_k}_{\phi_1,\phi_2,\cdots,\phi_k}$.
	\end{proof}
	\begin{remark}
		In fact, $t_i$ is not necessarily a fixed real number. It could be a continuous function  $t_i:(0,1]\rightarrow \mathbb{R}_{> 0}$, $\phi_{i+1}\mapsto t_{i}(\phi_{i+1})$, such that $t_{i}(1)\in\mathbb{Q}_{> 0}$. We assume them to be fixed numbers for the simplicity of the notation.
	\end{remark}
	\begin{definition}\label{l-thlevel definition}
		Given $t_1,t_2,\cdots,t_l\in\mathbb{Q}_{>0}$, suppose $E$ is $\sigma^{t_1}$ semistable of phase $\phi_1$ and $E\in\mathcal{A}^{t_1,t_2,\cdots,t_{k-1}}_{\phi_1,\phi_2\cdots,\phi_{k-1}}$ is semistable with respect to $\sigma_{\phi_1,\phi_2,\cdots,\phi_{k-1}}^{t_1,t_2,\cdots,t_{k}}$ of phase $\phi_k$ for any $1<k\leq l$. Then we say that $E$ is $l$-th level semistable of phase $\vec{\phi}=(\phi_1,\cdots,\phi_l)$, where $\phi_i\in(0,1]$.
		
		\begin{example}
			Note that in the trivial case $X=Spec(\mathbb{C})$, $\sigma=(\mathrm{Vect}, i\cdot dim)$, where $\mathrm{Vect}$ is the category of finite dimensional $\mathbb{C}$ vector spaces. And $S$ is a smooth projective variety of dimension $r$. Then being $1$-st level semistable is equivalent to being slope semistable, and being $(r+1)$-th level semitable is equivalent to being Gieseker semistable.
		\end{example}
		
		We will define the lexicographic order of such vectors, i.e., we say that $$(\phi_1,\phi_2,\cdots,\phi_l)>(\psi_1,\psi_2,\cdots,\psi_l)$$ if $\phi_1>\psi_1$, or $\phi_1=\psi_1,  \phi_2=\psi_2, \cdots, \phi_k=\psi_k$ and $\phi_{k+1}>\psi_{k+1}$ for some $1\leq k\leq l-1$.
	\end{definition}

	\begin{lemma}\label{invariance of phase}
		Given $t_1,t_2,\cdots,t_k\in\mathbb{Q}_{>0}$, assume that  $E$ is $(k-1)$-th level semistable of phase $(\phi_1,\phi_2,\cdots,\phi_{k-1})$. Let $$0=E_0\subset E_1\subset\cdots\subset E_n=E$$ be its HN filtration with respect to $\sigma_{\phi_1,\phi_2,\cdots,\phi_{k-1}}^{t_1,t_2,\cdots,t_k}$. Then every HN factor $E_i/E_{i-1}$ is $k$-th level semistable of phase $(\phi_1,\cdots,\phi_{k-1},\phi_{k,i})$, where $\phi_{k,1}>\phi_{k,2}>\cdots>\phi_{k,n}$ is a decreasing sequence of real numbers.
	\end{lemma}
	
	\begin{proof}
		We let $Q_i$ denote the HN factor $E_i/E_{i-1}$. Then $Q_i$ is semistable with respect to $\sigma_{\phi_1,\phi_2,\cdots,\phi_{k-1}}^{t_1,t_2,\cdots,t_k}$ of phase $\phi_{k,i}$.
		
		We claim that: for any object $F\in\mathcal{A}_{\phi_1,\cdots,\phi_k}^{t_1,\cdots, t_k}$ which is semistable with respect to $\sigma_{\phi_1,\cdots,\phi_k}^{t_1,\cdots,t_{k+1}}$, $F$ is also semistable with respect to $\sigma_{\phi_1,\cdots,\phi_{k-1}}^{t_1,\cdots,t_{k}}$ of phase $\phi_k$. The lemma follows naturally from the claim.
		
		Indeed, by definition of $\mathcal{A}^{t_1,\cdots, t_k}_{\phi_1,\cdots,\phi_k}$, we have the following short exact sequence $$0\rightarrow F_0\rightarrow F\rightarrow F'\rightarrow 0$$ where $F_0$ is in the abelianizer of $\sigma_{\phi_1,\cdots,\phi_{k-1}}^{t_1,\cdots,t_{k}}$ and $F'$ is semistable with respect to $\sigma_{\phi_1,\cdots,\phi_{k-1}}^{t_1,\cdots,t_{k}}$ of phase $\phi_k$. 
		
		If $F_0\simeq 0$, the claim is true. So now we assume that $F_0$ is a nonzero object. By the proof of Lemma \ref{slice stability condition}, we know that $$b_j(F_0)=a_j(F_0)=\cdots=b_{j-k+1}(F_0)=a_{j-k+1}(F_0)=b_{j-k}(F_0)=0.$$
		
		Hence the slope of $F_0$ with respect to $\sigma_{\phi_1,\cdots,\phi_k}^{t_1,\cdots,t_{k+1}}$ is $+\infty$. By the assumption $F\in\mathcal{A}_{\phi_1,\cdots,\phi_k}^{t_1,\cdots,t_k}$ is semistable with respect to $\sigma_{\phi_1,\cdots,\phi_k}^{t_1,\cdots,t_{k+1}}$ of phase $\phi_k$, we get that the imaginary part of $F$ with respect to $\sigma_{\phi_1,\cdots,\phi_k}^{t_1,\cdots,t_{k+1}}$ should also be $0$. By the proof of Lemma \ref{slice stability condition}, we know that $F$ is in the abelianizer of $\sigma_{\phi_1,\cdots,\phi_{k-1}}^{t_1,\cdots,t_{k}}$. Hence the claim is true and in this case we have $\phi_1=\phi_2=\cdots=\phi_{k+1}=1$.
	\end{proof}
	
	\begin{remark}
		From the proofs of Lemma \ref{slice stability condition} and Lemma \ref{invariance of phase}, we know that if $E$ is $k$-th level semistable of phase $(\phi_1,\cdots,\phi_k)$, then $\phi_i=1$ implies that $E$ is in the abelianizer of $\sigma_{\phi_1,\cdots,\phi_{i-1}}^{t_1,\cdots,t_{i}}$ and hence $\phi_h=1$ for all $1\leq h\leq i$.
	\end{remark}
	\begin{theorem}\label{lexi-order filtration}
		Assume $l\leq j$, given $t_1,t_2,\cdots,t_l\in\mathbb{Q}_{>0}$, for any object $E\in\mathcal{A}_S^{\leq j}$, there exists a  unique filtration $$0=E_0\subset E_1\subset E_2\cdots \subset E_n=E,$$such that each quotient factor $E_{i}/E_{i-1}$ is $l$-th level semistable of phase $\vec{\phi}_i$ and $$\vec{\phi}_1>\vec{\phi}_2>\cdots>\vec{\phi}_n$$.
	\end{theorem}
	
	\begin{proof}
		For any $k<l$, we claim that if $E$ is $k$-th level semistable of phase $\vec{\phi}=(\phi_1,\phi_2\cdots,\phi_k)$. There exists a filtration for $E$ as in the statement of theorem. Moreover, the first $k$ coordinates of all the vectors $\vec{\phi_1}>\vec{{\phi_2}}>\cdots>\vec{\phi_n}$ are equal to $\vec{\phi}$.
		
The claim can be shown by induction. Firstly, if $E$ is $(l-1)$-th level semistable of phase $(\phi_1,\phi_2,\cdots,\phi_{l-1})$. Then by Lemma \ref{invariance of phase}, we get the claim. 

 Assume that the claim is true for any $(l-i)$-th level semistable $E$. The goal is to prove the claim for any $(l-i-1)$-th level semistable object $E$. 

Now assume that $E$ is $(l-i-1)$-th level semistable of phase $(\phi_1,\phi_2,\cdots,\phi_{l-i-1})$, take its HN filtration with respect to $\sigma_{\phi_1,\phi_2,\cdots,\phi_{l-i-1}}^{t_1,t_2,\cdots,t_{l-i}}$. By induction on the number of HN factors, we can reduce to the case when there are only two HN factors. 

Indeed, suppose we have the following short exact sequence $$0\rightarrow K\rightarrow E\xrightarrow{f} Q\rightarrow0,$$ where $K, Q$ are $(l-i)$-th semistable of phases $$(\phi_1,\phi_2,\cdots,\phi_{l-i-1},\phi_{l-i,1}),\ (\phi_1,\phi_2,\cdots,\phi_{l-i-1},\phi_{l-i,2})$$ respectively. 

By induction, there exist filtrations $$0=K_0\subset K_1\subset \cdots\subset K_{n_1}=K$$ and 
$$0=Q_0\subset Q_1\subset \cdots\subset Q_{n_2}=Q$$ satisfying the requirements in the claim.

Now we take $E_{i}=K_{i}$ for $0\leq i\leq n_1$ and $E_{n_1+h}=f^{-1}(Q_h)$ for $0\leq h\leq n_2$, where $f^{-1}(Q_h)$ is the pullback of $Q_h$ along the map $f$. Then it is easy to check the filtration $$0=E_0\subset E_1\subset \cdots \subset E_{n_1+n_2}=E$$ satisfies the requirements of the claim. Hence we proved the existence of the filtration.

The uniqueness follows from the uniqueness of HN filtration. In fact, if we have a filtration $$0=E_0\subset E_1\subset E_2\cdots \subset E_n=E,$$such that each quotient factor $E_{i}/E_{i-1}$ is $l$-th level semistable of phase $\vec{\phi}_i$ and $$\vec{\phi}_1>\vec{\phi}_2>\cdots>\vec{\phi}_n,$$ then there is a partition $$S_1=\{\vec{\phi}_1,\vec{\phi}_2,\cdots,\vec{\phi}_{i_1}\},S_2=\{\vec{\phi}_{i_1+1},\cdots,\vec{\phi}_{i_2}\},\cdots,S_{k+1}=\{\vec{\phi}_{i_{k}+1},\cdots, \vec{\phi}_{n}\}$$ of the set $\{\vec{\phi}_1,\cdots,\vec{\phi}_n\}$ such that all the vectors in $S_i$ have the same first coordinate $\psi_i$ and $\psi_1>\psi_2>\cdots>\psi_{k+1}.$ 

By the proof of Proposition \ref{abelianizer}, we know that $$0=E_0\subset E_{i_1}\subset E_{i_2}\subset\cdots \subset E_{i_{k}}\subset E$$ is the HN filtration of $E$ with respect to $\sigma^{t_1}$. Hence it is unique. Then by inductively partitioning the vectors according to their next coordinates, we get the uniqueness of the filtration.
	\end{proof}

	We can use Theorem \ref{lexi-order filtration} to get finer torsion pairs of $\mathcal{A}_S$.
	
\begin{corollary}\label{finer cuts}
	Given $\vec{t}=(t_1,\cdots, t_{r+1})\in\mathbb{Q}_{>0}^{r+1}$ and $\vec{\phi}=(\phi_1,\cdots,\phi_{r+1})\in(0,1)^{r+1}$. Let $$\mathcal{T}_{\vec{\phi}}^{\vec{t}}\coloneqq \{E\in\mathcal{A}_S|all \ the \ factors \ in \ theorem\ \ref{lexi-order filtration}\ have \ phases \ > \vec{\phi}\},$$
	
	$$\mathcal{F}_{\vec{\phi}}^{\vec{t}}\coloneqq\{E\in\mathcal{A}_S|all \ the \ factors \ in \ theorem\ \ref{lexi-order filtration}\ have \ phases \leq \vec{\phi}\}.$$
	This pair is a torsion pair of $\mathcal{A}_S$.
	
\end{corollary}

\begin{proof}
	This follows directly from Theorem \ref{lexi-order filtration}. 
\end{proof}
	
	Hence we have a tilted heart $\mathcal{A}_{S,\vec{\phi}}^{\vec{t}}\coloneqq \langle \mathcal{F}_{\vec{\phi}}^{\vec{t}}[1], \mathcal{T}_{\vec{\phi}}^{\vec{t}}\rangle$. Take $\vec{\phi}=(\frac{1}{2},\cdots,\frac{1}{2})$, and denote the tilted heart by $\mathcal{A}_S^{t_1,\cdots, t_{r+1}}$.
	
	\begin{theorem}\label{positivity on tilted heart}
		For any object $E$ in $\mathcal{A}_S^{t_1,t_2,\cdots,t_{r+1}}$, we have the following inequalities.
		
		(1) $-a_r(E)t_1+b_{r-1}(E)\geq 0$.

		(2) For any positive integer $0< l\leq r$, if $$-a_r(E)t_1+b_{r-1}(E)=\cdots=-a_{r+1-l}(E)t_{l}+b_{r-l}(E)=0,$$ then we have $-a_{r-l}(E)t_{l+1}+b_{r-l-1}(E)\geq 0$. \footnote{Recall the convention $b_{-1}=0$.}
	\end{theorem}
	
	\begin{proof}
		By the definition of $\mathcal{A}_S^{t_1,t_2,\cdots,t_{r+1}}$, we have the following short exact sequence in $\mathcal{A}_S^{t_1,t_2,\cdots,t_{r+1}}$.
		
		$$0\rightarrow F[1]\rightarrow E\rightarrow T\rightarrow 0,$$ where $F\in\mathcal{F}_{\vec{\phi}}^{\vec{t}}$ and $T\in\mathcal{T}_{\vec{\phi}}^{\vec{t}}$, where $\vec{\phi}=(\frac{1}{2},\cdots,\frac{1}{2})$ and $\vec{t}=(t_1,t_2,\cdots,t_{r+1})$.

		By induction, we can easily show that if $$-a_r(E)t_1+b_{r-1}(E)=\cdots=-a_{r+1-l}(E)t_{l}+b_{r-l}(E)=0,$$ then we have $T, F\in\mathcal{A}_{\frac{1}{2},\cdots,\frac{1}{2}}^{t_1,t_2,\cdots,t_l }$. Hence by the definition of $\mathcal{F}_{\vec{\phi}}^{\vec{t}}$ and $\mathcal{T}_{\vec{\phi}}^{\vec{t}}$, we have $$-a_{r-l}(T)t_{l+1}+b_{r-l-1}(T)\geq 0$$ and $$-a_{r-l}(F)t_{l+1}+b_{r-l-1}(F)\leq 0.$$
		
		This completes the proof.
	\end{proof}
	
	\begin{remark}
	We choose the vector $\vec{\phi}$ to be $(\frac{1}{2},\cdots,\frac{1}{2})$ just for the simplicity of the statement. One can easily generalize the statement for arbitrary vector $\vec{\phi}\in(0,1]^{r+1}$.
	\end{remark}

\section*{Conflict of Interest}   The author declares that there is no conflict of interest.

\bibliographystyle{alpha}
\bibliography{bibfile}

\end{document}